\documentclass[a4paper, 10pt]{amsart}

\usepackage{
 amsmath
,amsfonts
,amssymb
,faktor
,tikz
,tikz-cd
,epigraph
,datetime
,xspace
,kodi
,mathtools
,hyphenat
,lmodern
,todonotes
,wasysym
}

\newcommand{\wco}{weak classifying object\@\xspace}

\def\quot{\textsf{Q}}
\def\sub{\textsf{S}}

\def\M{\mathcal{M}}
\def\K{\mathcal{K}}

\usepackage[%
   headheight=14pt%
  ,left=45mm%
  ,right=45mm%
  ,top=45mm%
  ,bottom=55mm]{geometry}

\def\opp{\text{op}}
\providecommand{\abbrv}[1]{#1.\@\xspace}
\providecommand{\ie}       {\abbrv{i.e}}

\providecommand{\adef}     {Definition\@\xspace}

\providecommand{\athm}     {Theorem\@\xspace}

\usepackage[all,cmtip]{xy}

\usepackage[cal=boondoxo]{mathalfa}

\def\xto#1{\xrightarrow{#1}}
\def\xot#1{\xleftarrow{#1}}

\usepackage{amsthm}

\usepackage{CJKutf8}

\newtheoremstyle{reference}
   {}                
   {}                
   {}              
   {}                      
   {\fontseries{b}\selectfont}              
   {:}                     
   {.2em}                  
   {\thmname{#1}           
    \thmnumber{#2}         
    \thmnote{{\sc [#3]}}} 

\def\wk{\textsc{wk}}
\def\cof{\textsc{cof}}
\def\fib{\textsc{fib}}

\theoremstyle{reference}
  \newtheorem{theorem}{Theorem}[section]
  \newtheorem{lemma}[theorem]{Lemma}
  \newtheorem{proposition}[theorem]{Proposition}
  \newtheorem{example}[theorem]{Example}
  
  \newtheorem{remark}[theorem]{Remark}
  \newtheorem{definition}[theorem]{Definition}
  \newtheorem{corollary}[theorem]{Corollary}
  \newtheorem{notat}[theorem]{Notation}

  \newtheorem*{theorem*}{Theorem}
  \newtheorem*{lemma*}{Lemma}
  \newtheorem*{proposition*}{Proposition}
  \newtheorem*{example*}{Example}
  \newtheorem*{exercise*}{Exercise}
  \newtheorem*{remark*}{Remark}
  \newtheorem*{definition*}{Definition}
  \newtheorem*{corollary*}{Corollary}
  \newtheorem*{notat*}{Notation}
  \newtheorem*{scholium*}{Scholium}
  \newtheorem*{counterex*}{Counterexample}
  \newtheorem*{conjec*}{Conjecture}
  \newtheorem*{quest*}{Question}

\renewcommand{\textbf}[1]{\text{\fontseries{b}\selectfont{\upshape #1}}}

\newcommand{\cate}[1]{\textbf{#1}}
\newcommand{\A}{\cate{A}}

\newcommand{\C}{\cate{C}}

\newcommand{\Top}{\cate{Top}}
\newcommand{\Cat}{\cate{Cat}}
\newcommand{\Gpd}{\cate{Gpd}}
\newcommand{\sSet}{\cate{sSet}}
\newcommand{\Set}{\cate{Set}}

\newcommand{\Sets}{\Set}

\newcommand{\QZ}{\mathbb{Q}/\mathbb{Z}}

\newlength{\seplen}
\newlength{\sepwid}
\def\firstblank{\,\rule{\seplen}{\sepwid}\,}

\def\[{\begin{equation}}
\def\]{\end{equation}}

\newcommand{\W}{\mathcal{W}}
\newcommand{\ho}{\textsc{ho}}

\usepackage{hyperref}
\hypersetup{%
  pdftoolbar=   true,
  pdfmenubar=   true,
  pdffitwindow= true,
  pdftitle=     {},
  pdfauthor=    {},
  colorlinks=   true,
  linkcolor=    black,
  citecolor=    blue!40!black}

\providecommand{\refbf}[1]{\textbf{\ref{#1}}}
\makeatletter
  \def\@cite#1#2{[\textbf{#1}\if@tempswa , #2\fi]}
  \def\@biblabel#1{[\textsf{#1}]}
\makeatother

\def\signed #1{{\leavevmode\unskip\nobreak\hfil\penalty50\hskip2em
  \hbox{}\nobreak\hfil(#1)%
  \parfillskip=0pt \finalhyphendemerits=0 \endgraf}}

\newsavebox\mybox
\newenvironment{aquote}[1]
  {\savebox\mybox{#1}\begin{quotation}}
  {\signed{\usebox\mybox}\end{quotation}}

\newcommand{\pullback}[2]{\node at ($(#1)!.25!(#2)$) {$\lrcorner$};}
\newcommand{\pushout}[2]{\node at ($(#1)!.75!(#2)$) {$\ulcorner$};}

\usepackage{cjhebrew}
\def\coker{\text{coker}}
\makeatletter
\def\@settitle{\begin{center}%
  \baselineskip14\p@\relax
  \bfseries
  \uppercasenonmath\@title
  \@title
  \ifx\@subtitle\@empty\else
     \\[1ex]\uppercasenonmath\@subtitle
     \footnotesize\mdseries\@subtitle
  \fi
  \end{center}%
}
\def\subtitle#1{\gdef\@subtitle{#1}}
\def\@subtitle{}
\makeatother

\begin{document}

\title{Homotopical algebra is not concrete}

\author{Ivan di Liberti}
\author{Fosco Loregian$^\dag$}
\thanks{$^\dag$ The second author is supported by the Grant Agency of the Czech Republic under the grant \textsc{P}201/12/\textsc{G}028.}
\address{
$^\dag$Department of Mathematics and Statistics\newline
Masaryk University, Faculty of Sciences\newline
Kotl\'{a}\v{r}sk\'{a} 2, 611 37 Brno, Czech Republic\newline
\href{mailto:diliberti@math.muni.cz}{\sf diliberti@math.muni.cz}\newline
\href{mailto:loregianf@math.muni.cz}{\sf loregianf@math.muni.cz}
}

\date{\today}
\maketitle 

\begin{abstract}
We generalize Freyd's well-known result that ``homotopy is not concrete'',
offering a general method to show that under certain assumptions on a model
category $\M$, its homotopy category $\ho(\M)$ cannot be concrete. 
This result is part of an attempt to understand more deeply the relation between 
set theory and abstract homotopy theory.
\end{abstract}
\section{Introduction}
\begin{aquote}{\cite{fconc}}
[The homotopy category of spaces $\cate{Ho}$] has always been the best example
of an \emph{abstract} category -- though its objects are spaces, the points of
the spaces are irrelevant because the maps are not functions -- best, because of
all abstract categories it is the one most often lived in by real
mathematicians. \emph{It is satisfying to know that its abstract nature is
permanent, that there is no way of interpreting its objects as some sort of set
and its maps as functions.}
\end{aquote}
As final as it may sound, Freyd's result that ``homotopy is not concrete'', and
in particular the paragraph above, doesn't address the fundamental problem of
\emph{how often} and \emph{why} the homotopy category of a category $\C$ endowed
with a class $\W_\C \subseteq \hom(\C)$ of weak equivalences is not concrete.

One of the strongest motivations in writing the present paper has been to fill
this apparent gap in the literature, clarifying which assumptions on a (model or
relative) category $(\M, \wk)$ give the homotopy category $\ho(\M) =
\M[\wk^{-1}]$ the same permanently abstract nature.

Our main claim here is that indeed Freyd's theorem generalizes quite easily,
and that many model categories that naturally arise in practice 
do not have a concrete localization at weak
equivalences; moreover, in light of our theorem the reason why this happens is
now evident. If we say, somewhat sloppily, that a category $(\M,\wk)$ is
`homotopy\hyp{}concrete' when $\ho(\M)$ is concrete, our result can be
summarized as the statement that very few model categories are
homotopy\hyp{}concrete, and that this somehow happens \emph{as a consequence} of the
fact that they encode an homotopy theory.

It is of course possible, at least in certain cases, to show that a given $\M$
is not homotopy\hyp{}concrete using ad-hoc arguments adapted to the particular
choice of the pair $(\M, \wk)$: in \cite{fconc} Freyd does this for the category
$\Cat$ with its `folk' model structure in which $\wk$ is the class of equivalences
of categories. Apart from being quite involved, Freyd's approach fails
to put the result, and similar others, on the same conceptual ground. Pursuing such a 
structural approach is the main aim of the present work.

Our main theorem, proved at page \pageref{ginnunga}, is
\refbf{ginnunga}:
\begin{theorem*}
Let $\M$ be a pointed model category; if there exist an index $n_0 \in \mathbb{N}_{\ge
1}$ and a `weak classifying object' for the functor $\pi_{n_0} \colon \M \to
\cate{Grp}$ (\adef\refbf{wcodef}), then $\M$ is not homotopy\hyp{}concrete.
\end{theorem*}
Freyd's argument is a completely formal construction relying on nifty, elementary algebraic construction in abelian group theory (Lemma \refbf{spastic}) and on the fact that the category
of spaces ``contains a trace'' of the category of abelian groups, via the
\emph{Moore functors} $M(\firstblank,n) \colon \cate{Ab}\to \Top$. In our generalization to model categories, we rely on similar properties of Eilenberg-Mac Lane-like objects (we call them \emph{weak classifying objects} in \autoref{wcodef}), and their interplay with the looping functor $\Omega$.

We are then able to apply the machinery of \athm\refbf{ginnunga} to several
explicit examples, thus showing that Freyd`s claim that ``homotopy is not
concrete'' remains true in the modern parlance of homotopical algebra. This
suggests how the permanent abstractness of homotopy theory is a reflection of
the permanent abstractness of \emph{homotopical algebra}.

In more detail, as a consequence of \athm\refbf{ginnunga} we offer
\begin{itemize}
	\item a proof that the homotopy category of chain complexes is not homotopy concrete;
	\item a proof that the homotopy category of $\Cat_\text{folk}$\footnote{As
already mentioned, this is a shorthand to refer to the category of small
categories with its `folk' model structure having weak equivalences the
equivalences of categories, and cofibrations the functors injective on objects.}
is not concrete, independent from (and surely more elegant than) the argument
presented in \cite[§4.1]{fconc}; here the result is a corollary of the fact
that the category of groupoids is not homotopy concrete, and this, in turn,
follows from the fact that the category of \emph{1-types} is not homotopy
concrete (these two categories being Quillen equivalent).
	\item A proof that the stable category of spectra $\cate{Sp}$ is not homotopy
concrete. Freyd \cite{Freydconc} observes that the stable category obtained as
Spanier-Whitehead stabilization of \textsc{cw}-complexes of dimension $\ge 3$
can't be concrete; our \refbf{spectra} can be thought as a slight refinement
that makes no assumptions on dimension.
	\item A proof that the local model structure
\cite{jardine1987simplical,dugger2004hypercovers} on the category of simplicial
sheaves on a site is not homotopy concrete.
\end{itemize}
Of course, we do not see these results as unexpected, given the tight relation
between unstable and stable homotopy, between categories and (geometric
realization of) simplicial sets, and between algebraic topology and algebraic
geometry.

We feel this is an additional step towards a deeper understanding of the notion
of concreteness and foundational issues in homotopy theory, and an additional
hint, if needed, for how set theory and homotopy theory do (or do not)
interact.
\section{Generalities on concreteness}
We recall the main definition we will work with (see \cite{Bor1,McL}):
\begin{definition}\label{concrecat} A category $\C$ is called \emph{concrete} if
it admits a faithful functor $U\colon \C \to \Sets$.
\end{definition} Concreteness can be regarded as a smallness request; in fact,
the following remark shows that many of the categories arising in mathematical
practice are concrete simply because they are not big enough, either because
they are small, or because they are accessible (the proof of each of the
following statements is easy).
\begin{remark}[almost everything is concrete]\label{all-is-conc} Every small
category is concrete. Every accessible category is concrete. If a category is
not concrete, none of its small subcategories can be dense. Concreteness is a
self-dual property, \ie $\C$ is concrete if and only if $\C^\opp$ is concrete.
If $\C$ is a concrete, and $J$ is a small category, then the functor category
$\C^J$ is concrete. If $\C$ is monadic over $\Set$, then it is concrete.
\end{remark} 
The paper \cite{Isbell1964} states a condition for the concreteness of a category $\C$, which relies on the notion of a \emph{resolvable relation} between the classes of spans and cospans in $\C$. This was linked by \cite{freyd1973concreteness} to a smallness request on the class of so-called \emph{generalized regular subobjects} of objects in $\C$. In more detail, \cite{freyd1973concreteness} proves the following statements:
\begin{itemize}
	\item the Isbell condition of \cite{Isbell1964} is equivalent, in a category with finite products, to the smallness of the class of \emph{generalized regular subobjects} of each object $A\in\C$, that we define below in \refbf{grsdef}.
	\item In a category with finite limits, the Isbell condition is equivalent to the smallness of the class of regular subobjects of each object $A\in\C$.
	\item The smallness of each class of generalized regular subobjects is necessary for concreteness.
\end{itemize}
\begin{definition}[Regular Generalized Subobject]\label{grsdef} Let $f\in
\K_{/A}$ an object of the slice category, and let $C(f,B)$ be the class of pairs
$u,v : A \to B$ such that $uf=vf$. %$ as morphisms $\text{src}(f) \to A \to B$.
Define an equivalence relation $\asymp$ on objects of $\K_{/A}$ as
\begin{center}
$f\asymp g$ iff $C(f,B)=C(g,B)$ for every $B\in\K$, 
\end{center}
and let $\sub(\K_{/A})$ be the quotient of $\K_{/A}$ under this
equivalence relation. This is called the class of \emph{generalized regular
subobjects} of $A\in\K$.
\begin{remark}
Briefly, the $\asymp$ relation identifies two maps with codomain $A$ if and only if they equalize the same pairs of arrows. In the category of sets and functions, any morphism $f$ satisfies $f\asymp m_f$, where $m_f$ is the monomorphism appearing in the epi-mono factorization of $f$. More generally, the same argument shows that $f\asymp m_f$ in every category endowed with a factorization system with regular monomorphisms as right class. A slightly more general argument shows that in a finitely complete category with cokernel pairs, a morphism $f \colon X\to A$ is $\asymp$-equivalent to the regular monomorphism $q$ appearing in the equalizer
\[\notag
\begin{kodi}
\obj{E & A & |(push)| A\cup_X A \\};
\mor E q:-> A u:{swap,shove=4pt},-> push;
\mor A v:{shove=-4pt},-> push;
\end{kodi}
% \xymatrix{
% 	E \ar[r]^-q& A\ar@<4pt>[r]^-u\ar@<-4pt>[r]_-v & A\cup_XA.\\
% }
\]
In \cite{freyd1973concreteness} it is stated that in a category with finite limits the generalized regular subobjects of $A$ coincide with the regular subobjects of $A$ for each object $A$.
\end{remark}
\end{definition}
\begin{proposition}[Freyd condition]\label{isbell}
If $\K$ is concrete then its class of generalized regular subobjects $\sub(\K_{/A})$ is a set for every $A\in\K$.
\end{proposition}
It is worthwhile to notice that there is a completely dual definition of \emph{generalized regular quotients} $\quot(\K_{A/})$: one similarly defines a relation that identifies two maps with \emph{domain} $A$ if and only if they \emph{coequalize} the same pairs of arrows. The size of equivalence classes of generalized regular quotients characterize concreteness as well:
\begin{proposition}[co-Freyd condition]\label{coisbell}
If $\K$ is concrete then its class of generalized regular quotients $\quot(\K_{A/})$ is a set for every $A\in\K$.
\end{proposition}

\begin{remark}
Recall that if $\K$ has finite products, the Freyd condition is equivalent to the Isbell condition and thus to concreteness of $\K$. Since all the categories in this paper have finite products there is no real interest in distinguishing the two conditions. Instead of choosing cumbersome notation as Freyd-Isbell condition or similar, we conflate the two conditions, referring to the result, for the sake of brevity, as \emph{the Isbell condition}.
\end{remark}
Several universal constructions of $\Cat$ restrict to constructions on model categories: given the purpose of this work, we are principally interested in those constructions that transport non-concreteness. These includes particularly simple examples: equivalent categories are either both concrete or both non-concrete (so that every category which is Quillen equivalent to a given non-homotopy\hyp{}concrete one is non-homotopy concrete as well), and if $\mathcal{L}\hookrightarrow \K$ is a subcategory and $\mathcal{L}$ is not concrete, so is~$\K$.

We will sometimes exploit such straightforward results to prove that a model category $\M$ is not homotopy\hyp{}concrete. We need only a functor that is \emph{homotopy faithful}, meaning that it induces inclusion \emph{between localizations}; more than often there is no control on which maps $\mathcal{L}(X,Y)\to \M(X,Y)$ become monomorphisms $\ho(\mathcal{L})(X,Y)\to \ho(\M)(X,Y)$, so we need to single out a special case when this happens.
\begin{definition}[piercing model subcategory]
Let $\M$ be a model category; a \emph{piercing model subcategory} is a full subcategory $\W \overset{U}\hookrightarrow \M$, which is reflective and coreflective, and having the model structure for which an arrow $\varphi \colon W \to W'$ is in $\wk,\cof,\fib$ if and only if $U\varphi$ is in $\wk,\cof,\fib$ as an arrow of $\M$.
\end{definition}
\begin{remark}
In the terminology of \cite{may2011more}, a piercing model subcategory is a reflective and coreflective subcategory such that the inclusion $U$ \emph{strongly creates} the model structure on $\W$.
\end{remark}
\begin{definition}
Let $\W \hookrightarrow \M$ be a piercing model subcategory; we say that $\W$ is \emph{homotopy replete} if given a zig zag of weak equivalences in $\M$ from an object of $\W$, 
\[
W \leftrightarrow  \dots \leftrightarrow M
\]
the arrows, as well as the object $M$, lie in $\W$.
\end{definition}
\begin{proposition}
A piercing model subcategory $\W \overset{U}\hookrightarrow \M$ induces a faithful functor $\ho(\W) \overset{\ho(U)}\hookrightarrow \ho(\M)$.
\end{proposition}
\begin{proof}
Since $\W$ is piercing in $\M$, there is a commutative square
\[
\begin{kodi}
\obj{|(A)| \W(\tilde V,\hat W)& [6.5em] |(B)| \M(U\tilde V,U\hat W) \\ |(D)| \ho(\W)(\tilde V,\hat W) & |(C)| \ho(\M)(U\tilde V,U\hat W)\\};
\mor A -> B ->> C;
\mor[swap] * ->> D {\overline{U}_{VW}}:dashed,-> *;
\end{kodi}
\]
where the horizonal arrows are the actions of the functors $U,\ho(U)$ on hom-sets. The first isomorphism theorem for sets now yields that $\overline{U}_{VW}$ is injective.
\end{proof}
\begin{example}\label{gpd-in-cat}
The inclusion $\Gpd_\text{folk}\hookrightarrow \Cat_\text{folk}$ turns $\Gpd_\text{folk}$ into a piercing model subcategory.
\end{example}
\begin{remark}\label{unconcrete-frombelow}
As a consequence of this result, if a piercing model subcategory $\W \hookrightarrow \M$ is not homotopy concrete, then neither is $\M$.
\end{remark}

\section{$\cate{Ho}$ is not concrete}
The group of remarks in \refbf{all-is-conc} suggests that ``every'' category arising in mathematical practice should be concrete. And yet, in his \cite{fconc} Peter Freyd was able to offer a nontrivial example of a non-concrete category, consisting of topological spaces and homotopy classes of continuous functions.

Freyd's proof is based on several technical lemmas and it is in fact the result of an extremely clever manipulation of basic constructions on topological spaces. We now provide a short but detailed survey of his original idea.
Our aim in this section is to refurbish the classical proof of the theorem contained in \cite{fconc} that, spelled out in modern terms, asserts the following:
\begin{theorem}[Homotopy is not concrete]\label{honoconc}
Let $\wk$ denote the class of \emph{homotopy equivalences} in the category $\Top$ of topological spaces. Then the Gabriel-Zisman localization \cite{GZ} $\cate{Ho}=\Top[\wk^{-1}]$ is not concrete in the sense of \adef\refbf{concrecat}.
\end{theorem}
\begin{remark}
Freyd's proof seems to leave us free to choose \emph{any} model category structure on $\Top$ that has homotopy equivalences as weak equivalences (and, in particular, \cite{fconc} makes no mention of the classes of cofibrations), and even though such a model structure seems to exist \cite{strom1972homotopy}, the details of this proof are subject of debate (there's an instructive discussion on the $n$Lab page \cite{nlabstrommodel}). Therefore, we decide to restrict our attention to the subcategory of $\Top$ whose objects are \emph{compactly generated spaces}, where a more modern technology is available. We still denote this subcategory as $\Top$.
\end{remark}
\subsection{$\cate{Ho}$ is not concrete: the proof}
Freyd's strategy can be summarized in the following two points:
\begin{itemize}
	\item As stated above \refbf{grsdef}, concreteness for a category $\A$ is equivalent to \emph{Isbell condition}. This necessary and sufficient condition was first proved in \cite{Isbell1964} (necessity) and \cite{fconc} (sufficiency).
	\item In the homotopy category of spaces $\cate{Ho}$ it is possible to find an object (in fact, many) admitting a proper class of generalized regular subobjects.
\end{itemize}
For the remainder of the proof, we fix:
\begin{enumerate}
	\item An integer $n\ge 1$;
	\item an arbitrary prime $p$.
\end{enumerate}
To build an object with a proper class of generalized regular subobjects we manipulate the cofibration sequence of a suitable Moore space. The main tool here is a technical lemma that generates a proper class of groups having arbitrarily large height (see \cite{fuchs2015abelian}). 
\begin{lemma}[black box lemma]\label{spastic}
There exists a sequence $B_\bullet = (B_\alpha)$ of $p$-torsion abelian groups, one for each ordinal number $\alpha\in\cate{Ord}$, satisfying the following conditions:
\begin{itemize}
	\item each $B_\alpha$ contains an element $x_\alpha$ such that $p x_\alpha = 0$;
	\item when $\alpha < \beta$ every homomorphism of groups $f_{\alpha\beta} \colon B_\alpha \to B_\beta$ such that $f(px)=p f(x)$ sends $x_\alpha$ to zero.
\end{itemize}
\end{lemma}
We adopt this statement without further explanation (hence the name \emph{black box}): the interested reader will find a proof, based on the theory of \emph{heights} of torsion abelian groups, in \cite{Freydconc}.
\begin{notat}
For each ordinal $\alpha$, let now $M_\alpha$ be the Moore space $M(B_\alpha,n)$ on the group $B_\alpha$  in grade $n$ that we found inside the black box of Lemma \refbf{spastic}. Let $t_\alpha\colon \mathbb{Z}/p\mathbb{Z} \to B_\alpha$ be a group morphism having $x_\alpha$ in its image, and $u_\alpha\colon M \to M_\alpha$ the induced map $M(t_\alpha,n)$ between Moore spaces. We denote by $M$ the Moore space for $\mathbb{Z}/p\mathbb{Z}$ in degree $n$.
\end{notat}
\begin{remark}
Notice that in the canonical cofiber sequence
\[
M \xto{u_\alpha} M_\alpha \to C_\alpha \to \Sigma M \to \Sigma M_\alpha\to\dots
\]
the space $C_\alpha$ is a Moore space for $\coker(t_\alpha)$, and $\Sigma M_\alpha$ is a Moore space for $B_\alpha$. This is a key point in the proof.
\end{remark}
Now we claim that 
\[
\{C_{\alpha} \to \Sigma M \mid \alpha \in \cate{Ord}\}
\] is a proper class of generalized regular subobjects for $\Sigma M$. In order to prove this claim we need the following
\begin{proposition}
\label{key}
For each pair of ordinals $\alpha < \beta$ the composition
\[
C_\beta \xto{v_\beta} \Sigma M \xto{\Sigma u_{\alpha}} \Sigma M_\alpha
\]
is not null-homotopic.
\end{proposition}
\begin{proof}
We argue by contradiction: assume that the composition is homotopic to a constant map. Since $\Sigma M_\beta \simeq \text{cone}(v_\beta)$, we get a map $\Sigma M_\beta \to \Sigma M_\alpha$ that makes the left triangle below commute.
\[\label{the-diag}
\begin{tikzcd}
\Sigma M \ar[r]\ar[dr]& \Sigma M_\beta \ar[dotted, d]\\
& \Sigma M_\alpha
\end{tikzcd}
\qquad\qquad
\begin{tikzcd}
\mathbb{Z}/p\mathbb{Z} \ar[r]\ar[dr]& B_\beta \ar[dotted, d]\\
& B_\alpha
\end{tikzcd}
\]
But then applying the functor $H_{n+1}(\firstblank, \mathbb{Z})$ to this commutative triangle we get a contradiction on the right diagram of abelian groups whose solid arrows contain $x_\beta$ in their images, and yet the dotted arrow is the zero map on $x_\beta$.
\end{proof}
\begin{remark}
This is one of the most important remarks in the section. Until now our proof lived in $\Top$, but drawing diagram (\refbf{the-diag}), and in particular the arrow $\Sigma M_\beta \to \Sigma M_\alpha$, we have to move in the localization $\cate{Ho}$, as this arrow only exists there: in fact, there might be no map whatsoever between these two objects filling the triangle above, but only a zig-zag of continuous maps
\[
\Sigma M_\beta \xot{\simeq} \bullet \to \bullet \xot{\simeq} \bullet \dots \to \Sigma M_\alpha
\]
\end{remark}
Finally, we can conclude the proof.

\begin{proposition}
All the arrows $v_\alpha \colon C_\alpha \to \Sigma M$ form \emph{distinct} generalized regular subobjects of $\Sigma M$, so that $\sub(\cate{Ho}_{/\Sigma M})$ contains a proper class.
\end{proposition}
\begin{proof}
Suppose $v_\alpha \asymp v_\beta$ for $\alpha < \beta$. Since in the following diagram
\[
\begin{kodi}
\obj{
|(Ca)| C_\beta & &[3em] \\[-2em]
&|(SM)| \Sigma M &|(SMb)| \Sigma M_\alpha\\[-2em]
|(Cb)| C_\beta & & \\
};
\mor SM {\Sigma u_\alpha}:-> SMb;
\mor Ca v_\beta:-> +;
\mor[swap] Cb v_\alpha:-> SM;
\end{kodi}
\]
the composition of $\Sigma u_\alpha \circ v_\alpha$ is null-homotopic (we assumed that $v_\alpha$ and $v_\beta$ equalize the same arrows, hence they both equalize the pair $(0,\Sigma u_\alpha)$), also the composition $\Sigma u_\alpha \circ v_\beta$ is null-homotopic. This contradicts lemma \refbf{key}.
\end{proof}
This concludes Freyd's original proof, and paves the way to a certain number of questions and generalization.
\section{A criterion for unconcreteness}
The proof of \athm\refbf{honoconc} relies on
\begin{itemize}
\item the existence of Moore spaces;
\item the existence of the homology functors;
\item their interplay with the suspension functor $\Sigma : \Top \to \Top$.
\end{itemize}
Now, not every model category has a notion of homology, but in pointed model categories we can define \emph{homotopy groups} using the suspension-loop adjunction. Adapting Freyd's proof to this dual situation constitutes the original result of the present work. The discussion so far has been tailored to let the proof of our main theorem seem natural.
\subsection{Homotopy groups on model categories}
Homotopy groups with coefficients can be constructed in every pointed model category; for the record, we simply adapt the construction that \cite[§II.6]{Baues1989} performs in the more general setting of \emph{cofibration categories}.
\begin{definition}[$\Sigma\dashv \Omega$ adjunction, homotopy groups]\label{sigmomega}
Let $\M$ be a pointed model category,  one can define the \emph{suspension-loop adjunction} via the following diagrams that are, respectively, an homotopy pushout and an homotopy pullback in $\M$
\[
\begin{kodi}
\obj{
%	|(M)|\ho(\M) &[2em] |(pM)| \ho(\M) &[3em] |(pM')| \ho(\M) &[2em] |(M')| \ho(\M)\\[-2em]
	X &|(1a)| 0 & \Omega Y &|(1b)| 0\\
	|(1c)| 0 & \Sigma X &|(1d)| 0 & Y \\
};
\mor X -> 1a -> {Sigma X};
\mor * -> 1c -> *;
\mor {Omega Y} -> 1b -> Y;
\mor * -> 1d -> *;
%\mor M \Sigma:-> pM; \mor pM' \Omega:-> M';
\pushout{X}{Sigma X}
\pullback{Omega Y}{Y}
\end{kodi}
\]
This defines two functors that we denote $\Sigma : \ho(\M) \to \ho(\M)$ and $\Omega : \ho(\M) \to \ho(\M)$. It is easy to notice that $(\Sigma, \Omega)$ is an adjoint pair, since arrows $A\to \Omega B$ correspond bijectively to arrows $\Sigma A \to B$.
So we define the \emph{$n^\text{th}$ homotopy group of $X$ with coefficients in $A$} to be
\[
\pi^A_n(X) :=  \ho(\M)(A, \Omega^n X).
\]
\end{definition}
\begin{remark}\label{ocio-coef}
This definition is of course compatible with shifting, in the sense that, as a consequence of the adjunction $\Sigma\dashv \Omega$, we have a natural isomorphism 
$$\pi^A_n \circ \Omega  = \pi^A_{n+1}.$$ 
In our discussion coefficients will be hidden for notational simplicity. Of course this is a group when $n \geq 1$, abelian when $n\geq 2$.
\end{remark}

\subsection{The main theorem}
\begin{remark}
The idea of our theorem is that if we can realize a \emph{weak classifying object} $k(G)$, at least for each abelian group $B_\alpha$ in the black box lemma, then we can define a suitable object $K$ having a proper class of generalized regular quotients.

As already mentioned, an exact translation of Freyd's argument is impossible due to the lack of a `universal' homology theory on a general model category, and yet the main result is preserved with only a few adjustments:%. The main \emph{idea} is, indeed, absolutely unchanged in this translation procedure. We briefly outline the reasons why we are constrained to these adjustments:
\begin{itemize}
	\item First of all we must switch to generalized \emph{quotients}, since we work with homotopy groups, and this consequently forces us to play with \emph{looping} operations, and not suspensions; we have to consider \emph{fiber sequences} in the homotopy category of $\M$.
	\item Thus, we have to build a proper class of generalized regular quotient $K \to K_\alpha$ for some object $K$. 
	\item We can't rely on the existence of maps $B_\alpha \to \mathbb{Z}/p\mathbb{Z}$ that are nonzero on $x_\alpha$ for each $\alpha$ (this is because $B_\alpha$ contains a cyclic direct summand of order $p$ generated by $x_\alpha$, so there will always be a homomorphism $B_\alpha\to B_\beta$ sending $x_\alpha$ to $x_\beta$). Fortunately, the cyclic group $\mathbb{Z}/p\mathbb{Z}$ plays no special r\^ole in the proof; we can safely assume that each $B_\alpha$ has a group homomorphism $B_\alpha \to \QZ$ that does not vanish on $x_\alpha$ (and these always exist, since $\QZ$ is an injective abelian group).
\end{itemize}
\end{remark}
\begin{definition}[Weak classifying object]\label{wcodef}
Let $\M$ be a model category, and $\K$ any category. A \emph{weak classifying object} for $\M$, relative to a functor\footnote{The symbol $\varpi$ is an alternative glyph for the Greek letter $\pi$.} $\varpi \colon \ho(\M) \to \K$ is a functor $k\colon \K \to \ho(\M)$ such that
\begin{itemize}
	\item the composition $\varpi\circ k$ is a full functor;
	\item there is a natural transformation $\epsilon\colon \varpi\circ k \Rightarrow 1$ which is an objectwise epimorphism (in this case $k$ is a \emph{right} \wco), \emph{or} a natural transformation $\eta \colon 1\Rightarrow \varpi\circ k$ which is an objectwise monomorphism (in this case $k$ is a \emph{left} \wco).
\end{itemize}
\end{definition}
\begin{notat}
We speak of a \emph{weak classifying object} $k$ (without specifying a side) when it is irrelevant whether $k$ is a left or right weak classifying object. All arguments can be easily adapted according to this slight abude of notation.

When a model category $\M$ has a \wco relative to the functor $\pi_n \colon \ho(\M) \to \cate{Grp}$ (see \refbf{ocio-coef}) we say that it has a \emph{\wco of type $n$} and it will be denoted $k(\firstblank,n) \colon \cate{Grp}\to \ho(\M)$. The notion of a \wco is an abstraction, tailored to our purposes, of Eilenberg-Mac Lane spaces $G\mapsto K(G,n)$ on $\Top$. Of course, if $n\ge 2$, we feel free to restrict the domain of $k(\firstblank,n)$ to the category of abelian groups.
\end{notat}
\begin{remark}
Having found a \wco of type $n_0$ for some $n_0\ge 2$ entails that there is a \wco also for every other $\pi_m$, with $m>n_0$ (so $\ell = m-n_0 > 0$): the functor
\[
k(\firstblank,n_0+\ell) := \Omega^\ell \circ k(\firstblank,n_0)
\]
is a \wco of type $n_0+\ell$.
\end{remark}
\begin{remark}
There are at least three cases where $\M$ has a \wco of type $n$: like before, we denote $\pi_n^A$ the $n^\text{th}$ homotopy group functor (see \refbf{ocio-coef}) for $n\ge 1$ and a coefficient object $A$.
\begin{enumerate}
	\item When $\pi_n^A\colon \ho(\M)\to \cate{Ab}$ has a section (\ie there is a functor $K_n$ such that $\pi_n^A\circ K_n \cong 1$);
	\item When $\pi_n^A$ has a faithful left adjoint;
	\item When $\pi_n^A$ has a full right adjoint.
\end{enumerate}
\end{remark}
\begin{theorem}\label{ginnunga}
Let $\M$ be a pointed model category; if there exist a natural number $n\ge 2$ and a \wco of type $n$ for $\M$, then $\ho(\M)$ can not be concrete.
\end{theorem}
This proof occupies the rest of the section. We establish the following notation:
\begin{itemize}
	\item we only prove the statement in case the functor $k$ is a \emph{right} \wco; with straightforward modifications the proof can be easily dualized to the case of a left \wco with $\eta\colon 1 \Rightarrow \pi_n\circ k(\firstblank,n)$;
	\item the object $K_\alpha$ is the image of the group $B_\alpha$ in Lemma \refbf{spastic} via the functor $k(\firstblank,n)$; we also denote $B=\QZ$ and $K = k(\QZ,n)$;
	\item we fix a map $t_\alpha \colon B_\alpha \to \QZ$ such that $t_\alpha(x_\alpha)\neq 0$ (there is always such a $t_\alpha$ since $\QZ$ is an injective abelian group); we denote $u_\alpha = k(t_\alpha,n) \colon K_\alpha \to K$.
\end{itemize}
Now, consider the fiber sequence
\[
\dots \to \Omega K_\alpha \to \Omega K \xto{v_\alpha} F_{\alpha} \to K_{\alpha} \xto{u_\alpha} K 
\]
We will use the co-Isbell condition \refbf{coisbell} to prove that since
\[
(\Omega K \xto{v_\alpha} F_{\alpha})_{\alpha\in\cate{Ord}}
\]
is a proper class of generalized regular quotient for $\Omega K$, the category can't be concrete. In order to do this, we re-enact Lemma \refbf{key} in the following form using the loop functor $\Omega$ instead of $\Sigma$:
\begin{lemma}\label{peterkey}
For each pair of ordinals $\alpha < \beta$ the composition
\[
\Omega K_\beta  \xto{\Omega u_\beta } \Omega K \xto{v_\alpha} F_{\alpha}
\]
is not null-homotopic.
\end{lemma}
\begin{proof}
We argue by contradiction: assume that the composition above is nullhomotopic; since $\Omega K_\alpha \simeq \text{fib}(v_\alpha)$, we get a map $\Omega K_\beta \to \Omega K_\alpha$ in $\ho(\M)$ that makes the triangle 
\[\label{the-diag-2}
\begin{kodi}
\obj{\Omega K_\beta & \Omega K & F_\alpha \\
& \Omega K_\alpha & \\};
\mor {Omega K_beta} {\Omega u_\beta}:-> {Omega K} {v_\alpha}:-> {F_alpha};
\mor[swap] * \varphi:-> {Omega K_alpha} {\Omega u_\alpha}:-> {Omega K};
\end{kodi}
\]
commute. But then the $\pi_{n-1}(\firstblank)$ of this commutative triangle embeds into the following bigger diagram:
\[
\begin{kodi}[remove characters=_\{\}, expand=full]
\foreach \i/\j in {0/,120/\alpha,-120/\beta}{
	\obj at (\i:1.4) {\pi_n k(B_{\j},n)};
	\obj at (\i:3.5) {B_{\j}};
	}
\mor {pi n kBalpha n} {\epsilon_\alpha}:-> {Balpha};
\mor {pi n kBbeta n} {\epsilon_\beta}:-> {Bbeta};
\mor {pi n kBn} {\epsilon}:-> {B};
\mor[swap] {pi n kBbeta n} {\pi_{n-1}(\varphi)}:{bend left},-> {pi n kBalpha n} {bend left},-> {pi n kBn};
\mor {pi n kBbeta n} {bend right},-> {pi n kBn};
\mor[swap] {Bbeta} \psi:{bend left, dashed},-> {Balpha} t_\alpha:{bend left},-> {B};
\mor * t_\beta:{bend right},-> {B};
\end{kodi}
\]
Every subdiagram made by solid arrows commutes, and the dotted arrow exists by the fullness assumption on $\pi_n\circ k(\firstblank,n)$. 

Since the $\epsilon$ arrows are all epimorphisms the outer triangle commutes. But this is impossible, since $\psi$ sends $x_\beta$ to $0$, whereas $t_\beta$ does not.
\end{proof}
\begin{proposition}
All the arrows $v_\alpha \colon \Omega K \to F_\alpha$ form \emph{distinct} generalized regular quotient of $\Omega K$, so that $\quot(\cate{Ho}_{\Omega K/})$ contains a proper class.
\end{proposition}
\begin{proof}
Suppose $v_\alpha \asymp v_\beta$ for $\alpha < \beta$. Since in the following diagram
\[
\begin{kodi}
\obj{
&&|(Fa)| F_\beta\\[-2em]
|(OmegaKa)|\Omega K_\beta &|(OmegaK)| \Omega K & \\[-2em]
&&|(Fb)| F_\alpha\\};
\mor OmegaKa {\Omega u_\beta}:-> OmegaK;
\mor OmegaK v_\beta:-> Fa;
\mor[swap] * v_\alpha:-> Fb;
\end{kodi}
\]
the composition of $v_\beta \circ \Omega u_{\beta}$ is null-homotopic and we assumed that $v_\alpha$ and $v_\beta$ equalize the same arrows, also the composition $ v_\alpha \circ \Omega u_\beta$ is null-homotopic. This contradicts lemma \refbf{peterkey}.
\end{proof}
\subsection{Quasistable model categories}
\begin{definition}[Quasistability]
A pointed model category is \emph{quasistable} if  the comonad $\Sigma\Omega$ of the adjunction $\Sigma\dashv \Omega$ in \refbf{sigmomega} is full.
\end{definition}
\begin{remark}
This definition is a weakening of the stability property for $\M$, as in the stable case the comonad $\Sigma\Omega$ is full (in fact, it is an equivalence).
\end{remark}
In a quasistable model category our main theorem takes the following form:
\begin{theorem}
\label{qsginnunga}
If $\M$ is quasistable, and it has a generalized \wco for \emph{some} functor $\varpi\colon \ho(\M) \to \cate{Grp}$ such that $\varpi * \varepsilon$ is an objectwise epimorphism, then it is not homotopy concrete.
\begin{proof}
The unstable proof can be adapted in the following way: consider the sequence of groups $B_\bullet$ obtained in \refbf{spastic}, regarded as valued in $\cate{Ab}\subset \cate{Grp}$, and the diagram
\[
\tiny
\begin{kodi}[remove characters=_\{\}, expand=full,xscale=.75,yscale=.75]
\foreach \i/\j in {0/,120/\alpha,-120/\beta}{
\obj at (\i:1.5) {\varpi\Sigma\Omega K_{\j}};
\obj at (\i:3.5) {\varpi K_{\j}};
\obj at (\i:5.5) {B_{\j}};
}
\mor {varpi Sigma Omega Kalpha} -> {varpi Kalpha} -> Balpha;
\mor {varpi Sigma Omega Kbeta} -> {varpi Kbeta} -> Bbeta;
\mor {varpi Sigma Omega K} -> {varpi K} -> B;
\mor :[bend left] {varpi Sigma Omega Kalpha} -> {varpi Sigma Omega K};
\mor :[bend right] * -> {varpi Sigma Omega Kbeta} -> *;
\mor :[bend right] {varpi Kbeta} -> {varpi K};
\mor :[bend left] * \star:dashed,-> {varpi Kalpha} -> *;
\mor :[bend right] {Bbeta} t_\beta:-> {B};
\mor[swap]:[bend left] * \star\star:dotted,-> {Balpha} t_\alpha:-> *;
\end{kodi}
\]
obtained using the same notation of the unstable proof. The starred arrows appear thanks to the assumption of quasistability, and the same argument shows that there can't be no nullhomotopic sequence $\Omega K_\beta \to \Omega K \to F_\alpha$ for $\alpha < \beta$.
\end{proof}
\end{theorem}
\section{Examples}
\begin{example}[Example 0]
Obviously, if two model categories are Quillen equivalent one is homotopy concrete if and only if the other is, and this requires no theorem whatsoever. So, as a consequence of \refbf{honoconc} every category Quillen equivalent to $\Top$ cannot be homotopy concrete. 
\end{example}
\begin{example}[Chain complexes]\label{complessi}
The homotopy category $\ho(\text{Ch}(\mathbb{Z}))$ of chain complexes of abelian groups with its standard model structure is not concrete. 

In fact homology functors $H_n$ have a \wco, that is the complex having a given abelian group $G$ in degree $n$ and zeroes elsewhere. Since this category is quasi stable (in fact, stable), the homotopy category cannot be concrete by \athm\refbf{qsginnunga}.
\end{example}
\begin{example}[Spectra]\label{spectra}
The category $\ho(\Omega\text{-}\cate{Sp}))$ obtained localizing the category of (Bousfield\hyp{}Friedlander) spectra is not concrete. Indeed, the stable homotopy functor $\pi^\text{s}_0 \colon \cate{Sp} \to \cate{Ab}$ has a \wco given by the Eilenberg-Mac Lane construction $A\mapsto K(A,\firstblank)$. Again, since this category is stable, the homotopy category cannot be concrete by \athm\refbf{qsginnunga}.
\end{example}
\begin{example}[Simplicial sheaves on a site]\label{fasci}
Let $(\C, J)$ be a small Grothendieck site; a model for hypercomplete $\infty$-stacks is the following:
\begin{itemize}
	\item Consider the category $[\C^\opp, \sSet]_\text{proj}$, endowed with the projective model structure with respect to the Kan-Quillen model structure on $\sSet$; this is called the \emph{global model structure}.
	\item now consider the left Bousfield localization given by the equivalences with respect to \emph{homotopy sheaves}, obtained as follows: consider the compositions
	\begin{gather*}
	[\C^\opp,\sSet] \xto{\pi_{0,*}} [\C^\opp, \Set] \xto{(\firstblank)^+} \cate{Sh}(\C,J) \\
	[(\C_{/X})^\opp,\sSet] \xto{\pi_{n,*}} [(\C_{/X})^\opp,\cate{Grp}] \xto{(\firstblank)^+} \cate{Sh}_\Delta(\C_{/X},J)
	\end{gather*}
	where the rightmost functor is $J$-sheafification. This defines functors $\underline{\pi}_n$ called the \emph{homotopy sheaves} of a simplicial presheaf $F$. A morphism $\eta\colon F\to G$ is a \emph{local equivalence} if it induces isomorphisms $\underline{\pi}_n(\eta) \colon \underline{\pi}_n(F) \overset{\cong}\to \underline{\pi}_n(G)$ between homotopy sheaves in each degree.%, and local equivalences form a Bousfield localization of the \emph{local model structure} described in \cite{jardine1987simplical,dugger2004hypercovers}.
\end{itemize}
We claim that the local model structure turns $\M=[\C^\opp, \sSet]$ into a category which is not homotopy concrete. To prove this, it suffices to consider the functor $\varpi_n \coloneqq \Gamma \circ \underline{\pi}_n\colon \M \to \Set_{*/}$ (depending on $n$ this functor will take values in abelian groups), giving the global sections of the homotopy sheaves. The construction of Eilenberg-Mac Lane stacks $K(\firstblank,n)$ of \cite[§\textbf{2.2}]{toen2010simplicial} gives \wco{}s of type $n$.
\end{example}

\section{A long and instructive example: $1\text{-}\cate{types}$ and $\Cat$}
As it is well known, the homotopy category of groupoids endowed with its `folk' model structure inherited from $\Cat$ is equivalent to the homotopy category of 1-types (spaces with vanishing $\pi_{\ge 2}$), via the \emph{classifying space} and \emph{fundamental groupoid} functors. This result certainly has an interest for both category theorists and homotopy theorists.% (we advise \cite{camarena2013whirlwind} as a modern and pleasant introductory reading on this topic).

The present section is completely devoted to prove that none of the three categories of 1-types, groupoids, and categories has a concrete localization at its natural choice of weak equivalences; to prove this we will heavily rely on the above-mentioned equivalence between groupoids and 1-types; this gives an elegant proof that the homotopy category of $\Cat$ is not concrete (a result that \cite[§4.1]{fconc} obtains with weeping and gnashing of teeth).

We remark that there can't be a model structure on the category $1\text{-}\cate{types}_*$ of pointed spaces with vanishing $\pi_{\ge 2}$, since the category is --not even finitely-- cocomplete. This might appear as an issue, as it shows that the assumptions of \athm\refbf{ginnunga} are not minimal: the presence of a mere class of weak equivalences $\W$ and a pair of homotopical functors $(\varpi,k)$, one of which nicely interacts with some `looping' functor $\Omega$ and the other is a \wco for the first is sufficient to build an object of the homotopy category with too many quotients. In fact, one could be tempted to state \athm\refbf{ginnunga} in the more general setting of \emph{categories of fibrant objects} ($1\text{-}\cate{types}_*$ is such a category).%, or in the even more general setting of what might be called `Puppe categories' where we are given $(\W,\varpi,k,\Omega)$ as above.

% We feel that such a weakening of assumptions does not yield substantial improvement in the discussion, as the proof of a statement like
% \begin{theorem*}
% Let $\{\M,(\W,\varpi,k,\Omega)\}$ be a pointed Puppe category; if there exist an index $n\ge 2$ and a \wco of type $n$ for $\M$, then $\M[\W^{-1}]$ can not be concrete.
% \end{theorem*}
% would go in the same way as the proof of \athm\refbf{ginnunga} (it is worth to notice that we already mentioned, right before \adef\refbf{sigmomega}, how the definition of homotopy groups with coefficient works also in a co/fibration category). 
A deeper discussion on this issue (\ie, what \emph{minimal} assumptions make our main theorem true) will certainly be the subject of further investigations.
\begin{example}[$1\text{-}\cate{types}_*$ is not homotopy concrete]\label{uno-tipi}
The category $1\text{-}\cate{types}_*$ has no concrete localization at its class of weak equivalences (induced by the inclusion $1\text{-}\cate{types}_*\subset\Top$): the fundamental group functor $\pi_1\colon 1\text{-}\cate{types}_* \to \cate{Grp}$ has the classifying space $K(\firstblank,1)$ as a \wco.
\end{example}
It is worth to outline the argument completely; as already mentioned, there's no model structure on $1\text{-}\cate{types}_*$ but its structure of category with fibrant objects is enough to conclude that it is not homotopy concrete, as the pair of functors $(\pi_1,K(\firstblank,1))$ still does what is needed: in the same notation of \athm\refbf{ginnunga}, the object $\Omega K$ is a 0-type (hence a fortiori a 1-type), and the maps $\Omega K \to F_\alpha$ still form a proper class of distinguished generalized quotients of $\Omega K$.

Now we would like to deduce, from the fact that the category of \emph{pointed} 1-types is not homotopy concrete, the fact that the category $1\text{-}\cate{types}$ of \emph{unpointed} 1-types is not concrete. This seemingly easy result requires instead quite an involved argument, as it is in general impossible to deduce the homotopy non-concreteness of $\M$ from the homotopy non-concreteness of the model category $\M_{*/}$ of pointed objects in $\M$: in this particular case, however, it's easy to see that the functor $1\text{-}\cate{types}_* \to 1\text{-}\cate{types}$ injects the proper class $\mathsf{Q}_*(\Omega K)$ of pointed generalized regular quotients into the class $\mathsf{Q}(\Omega K)$ of unpointed ones.
\begin{corollary}
The category of groupoids with the choice of its `folk' model structure, is not homotopy concrete.%, since \cite{camarena2013whirlwind} it is Quillen equivalent to $1\text{-}\cate{types}$.
\end{corollary}
This is clear, in view of the above-mentioned equivalence between groupoids and (unpointed) 1-types.
\begin{corollary}
Since (cf. \refbf{gpd-in-cat}) $\Gpd_\text{folk}$ is an homotopy replete model subcategory of $\Cat_\text{folk}$, we conclude that $\Cat_\text{folk}$ can not be homotopy concrete.
\end{corollary}

\subsubsection*{Acknowledgements}
The authors would like to thank professor Dan Christensen for a preliminary and attentive reading of the first draft of this paper, professor Jiří Rosický for  his support in our investigation, professor Ivo Dell'Ambrogio for persuading us about the relevance of our result to a public of algebraic topologists, and more in general everybody who contributed to the improvement of this work.

\hrulefill

\bibliography{allofthem}{}
\bibliographystyle{amsalpha}
\end{document}